\documentclass[%
 jmp,
 amsmath,amssymb,
reprint, onecolumn %%%%% Comment onecolumn to override one column format %%%%%%%
]{revtex4-1}

\usepackage{graphicx}% Include figure files
\usepackage{dcolumn}% Align table columns on decimal point
\usepackage{bm}% bold math
\usepackage[utf8]{inputenc}
\usepackage[T1]{fontenc}
\usepackage{mathptmx}
\usepackage{etoolbox}
\usepackage{amsmath, amsthm, amsfonts, amssymb}

\makeatletter
\def\@email#1#2{%
 \endgroup
 \patchcmd{\titleblock@produce}
  {\frontmatter@RRAPformat}
  {\frontmatter@RRAPformat{\produce@RRAP{*#1\href{mailto:#2}{#2}}}\frontmatter@RRAPformat}
  {}{}
}%
\makeatother

\newcommand\Ex{{\mathbb E}}

\newcommand\Prob{{\mathbb P}}
\newcommand\chitilde{\tilde{\chi}}
\newcommand\Normal{{\mathcal N}}

\newcommand\cF{{\mathcal F}}

\newcommand\cL{{\mathcal L}}
\newcommand\cM{{\mathcal M}}

\newcommand\cP{{\mathcal P}}

\newcommand\R{{\mathbb R}}

\newcommand\C{{\mathbb C}}

\newcommand\dto{\overset{d}{\to }}
\newcommand\wto{\overset{w}{\to }}

\newcommand\norm[1]{\|#1\|}

\newcommand\bra[1]{\langle #1 \rangle}

\newcommand\as{\text{ as }}

\DeclareMathOperator{\Gam}{Gamma}

\DeclareMathOperator{\Beta}{Beta}

\DeclareMathOperator{\supp}{supp}

\newtheorem{theorem}{Theorem}[section]

\newtheorem{lemma}[theorem]{Lemma}
\newtheorem{proposition}[theorem]{Proposition}

\theoremstyle{definition}
\newtheorem{definition}[theorem]{Definition}

\theoremstyle{remark}
\newtheorem{remark}[theorem]{Remark}

\begin{document}

%\preprint{AIP/123-QED}

\title[Classical beta ensembles and the Markov--Krein transform]{Classical beta ensembles and related eigenvalues processes at high temperature and the Markov--Krein transform}
% Force line breaks with \\
\author{F. Nakano}
 %\email{fumihiko.nakano.e4@tohoku.ac.jp.}
 \affiliation{Mathematical Institute, Tohoku University, Sendai,  Japan.}%Lines break automatically or can be forced with \\

\author{H.D. Trinh}%
 %\email{thdung.hus@gmail.com.}
\affiliation{Faculty of Mathematics Mechanics Informatics, University of Science, Vietnam National University, Hanoi, Vietnam.%\\This line break forced with \textbackslash\textbackslash
}%

\author{K.D. Trinh}
\email{trinh@aoni.waseda.jp}
 %\homepage{http://www.Second.institution.edu/~Charlie.Author.}
\affiliation{%
Global Center for Science and Engineering, Waseda University, Japan.%\\This line break forced% with \\
}%

\date{\today}% It is always \today, today,

\begin{abstract}
The aim of this paper is to identify the limit in a high temperature regime of classical beta ensembles on the real line and related eigenvalue processes by using the Markov--Krein transform. We show that the limiting measure of Gaussian beta ensembles (resp.\ beta Laguerre ensembles and beta Jacobi ensembles) is the inverse Markov--Krein transform of the Gaussian distribution (resp.\ the gamma distribution and the beta distribution). At the process level, we show that the limiting probability measure-valued process is the inverse Markov--Krein transform of a certain $1$d stochastic process. 
\end{abstract}

\maketitle

\section{Introduction}
A probability measure $\nu$ on $\R$ and a probability measure $\mu$ on $\R$ with 
\(
	\int_\R \log(1 + |u|) d\mu(u) < \infty,
\)
are linked by the Markov--Krein relation (MKR) with parameter $c > 0$ if, for $z \in \C \setminus \R$,
\begin{equation}
	\int_\R \frac{1}{(z - t)^c} d\nu(t) = \exp\left(	- c \int_\R \log(z - u) d\mu(u)\right).
\end{equation}
Here the function $\log z$  and the function $z^\alpha$, for $\alpha \in \R$, are defined on $\C \setminus (-\infty, 0]$ as follows: if $z = r e^{i \theta}$, with $r > 0, -\pi < \theta < \pi$, then 
\[
	\log z = \log r + i \theta, \quad	z^{\alpha} = e^{\alpha \log z} = r^{\alpha} e^{i \alpha \theta}.
\]
Note that the relation between $\nu$ and $\mu$ uniquely determines each other \cite{Mergny-Potters-2022}. We call $\nu$ the Markov--Krein transform (MKT) of $\mu$ with parameter $c$ and write $\nu = \cM_{c}(\mu)$, and $\mu$ the inverse Markov--Krein transform (IMKT) of $\nu$ and write $\mu =\cM_{c}^{-1}(\nu)$. The MKT is related to Dirichlet processes \cite{CR90, Jam05}, representations of the symmetric group \cite{Ker03} and the Markov moment problem \cite{KN77}, to name a few. The case $c=1$ was intensively studied in a seminal work of Kerov \cite{Kerov-1988} in which the measure $\mu$ needs not be positive.

A motivated example for this work is related to the Dirichlet distribution. Let $w = (w_1, \dots, w_N)$ have the Dirichlet distribution with parameters $\tau = (\tau_1, \dots, \tau_N),  (\tau_i > 0)$, that is, a random vector $w = (w_1, \dots, w_N)$ belonging to the standard $(N-1)$ simplex, 
\(
	 \{w = (w_1, \dots, w_N) : \sum_{i=1}^N w_i = 1,  (w_i \ge 0)\},
\)
has the following joint probability density function with respect to the Lebesgue measure on $\R^{N-1}$
\[
	\frac{\Gamma(\tau_1 + \cdots + \tau_N)}{\prod_{i=1}^N \Gamma(\tau_i)}  \prod_{i=1}^N w_i^{\tau_i - 1}.
\]
For $a = (a_1, \dots, a_N) \in \R^N$, let $\nu$ be the distribution of $(w_1 a_1 + \cdots + w_N a_N)$, and $\mu = \frac{1}{|\tau|}\sum_{i=1}^N \tau_i \delta_{a_i}$ be a probability measure on $\R$. Here $|\tau| = \tau_1 + \cdots + \tau_N$ and $\delta_\lambda$ denotes the Dirac measure. Then $\nu$ and $\mu$ are linked by the MKR with parameter $c = |\tau|$, that is, 
\begin{equation}
	\int_\R \frac{1}{(z - t)^{c}} d\nu(t) = \Ex[(z - (w_1 a_1 + \cdots + w_N a_N))^{-c}] =  \prod_{i=1}^N \left(\frac1{z - a_i}\right)^{\tau_i}.
\end{equation}
The proof of the above relation can be found in \cite{Fourati-2011}.

In a particular case when $w = (w_1, \dots, w_N)$ has the symmetric Dirichlet distribution with parameter $c/N$ (that is, the Dirichlet distribution with parameters $(\frac{c}N, \dots, \frac cN)$), the measure $\mu$ coincides with the empirical distribution of $a$, 
\(
	\mu = L_N =  \frac1N \sum_{i=1}^N \delta_{a_i}.
\)
We call a random probability measure 
\(
	\nu_N = \sum_{i=1}^N w_i \delta_{a_i}
\)
the spectral measure of $a$. Then the above MKR can be expressed as
\begin{equation}\label{MKR-spectral-measure}
	\Ex[(z - \bra{\nu_N, x})^{-c}] = \exp(-c \bra{L_N, \log(z - x)}), \quad z \in \C \setminus \R.
\end{equation}
Here $\bra{\mu, f}$ denotes the integral $\int f d\mu$ of an integrable function $f$ with respect to a probability measure $\mu$.

Next we see how that example is related to classical beta ensembles at a high temperature regime. Three classical beta ensembles on the real line (Gaussian beta ensembles, beta Laguerre ensembles and beta Jacobi ensembles)  are families of joint probability density functions proportional to 
\begin{equation}\label{joint-density}
	 \prod_{i < j}|\lambda_j - \lambda_i|^\beta \prod_{l = 1}^N w(\lambda_l) ,\quad (\beta > 0),
\end{equation}
with weights $w$ given by
\begin{equation}\label{weights}
	w(\lambda) = 
	\begin{cases}
		e^{-\lambda^2/2},\quad &\lambda \in (-\infty, \infty), \quad \text{[Gaussian],}\\
		\lambda^{\alpha-1}e^{-\lambda}, \quad &\lambda \in (0, \infty),\quad (\alpha > 0),\quad \text{[Laguerre],}\\
		\lambda^a (1 - \lambda)^b, \quad &\lambda \in (0,1), \quad(a>-1,b>-1),\quad \text{[Jacobi]}.
	\end{cases}
\end{equation}
The parameter $\beta > 0$ is viewed as the inverse temperature parameter. Note that the weights are proportional to the density of the standard Gaussian distribution, the gamma distribution $\Gam(\alpha, 1)$ and the beta distribution $\Beta(a+1, b+1)$, respectively.
In a high temperature regime where $\beta = 2c/N$ with $c \in (0, \infty)$ being given, it is known that the empirical distribution of the eigenvalues $\{\lambda_i\}_{i=1}^N$,
\[
	L_N = \frac1N \sum_{i=1}^N \delta_{\lambda_i}
\] 
converges weakly to a probability measure $\mu_c$, almost surely. Equivalently, for any bounded continuous function $f$, 
\[
	\bra{L_N, f} \to \bra{\mu_c, f}, \quad \text{almost surely as } N \to \infty.
\]
The convergence also holds for continuous functions of polynomial growth \cite{Trinh-Trinh-Jacobi, Trinh-Trinh-2021, Trinh-2019}. 
The limiting probability measures  in the Gaussian, Laguerre and Jacobi cases are related to associated Hermite, Laguerre and Jacobi polynomials, respectively \cite{DS15, Trinh-Trinh-Jacobi, Trinh-Trinh-2021}.

The three classical beta ensembles are now realized as the eigenvalues of certain tridiagonal random matrices for any $\beta > 0$ \cite{DE02, Killip-Nenciu-2004}. That is to say, there are random matrices $J_N$ of the form
\[
	J_N = \begin{pmatrix}
		a_1	&b_1\\
		b_1	&a_2		&b_2\\
		&\ddots	&\ddots	&\ddots\\
		&&b_{N-1}	&a_N
	\end{pmatrix},\quad a_i \in \R, b_j > 0,
\]
called Jacobi matrices, such that their eigenvalues are distributed according to the joint density~\eqref{joint-density}.
The spectral measure of $J_N$ is defined to be the unique probability measure $\nu_N$ on $\R$ such that 
\[
	\bra{\nu_N, x^k} = (J_N)^{k}(1,1), \quad k = 0,1,\dots.
\]
It turns out that $\nu_N$ has a simple expression 
\[
	\nu_N = \sum_{i=1}^N w_i \delta_{\lambda_i},
\]
where $w_i = v_i(1)^2$ with $\{v_i\}_{i=1}^N$ the corresponding normalized eigenvectors of $J_N$. For all three ensembles, the vector of weights $w = (w_1, \dots, w_N)$ has the symmetric Dirichlet distribution with parameter $\beta/2 = c/N$, and is independent of the eigenvalues $(\lambda_1, \dots, \lambda_N)$.

Because of the independence, the relation~\eqref{MKR-spectral-measure} implies that
\begin{equation}\label{MKR-classical}
\Ex[(z - \bra{\nu_N, x})^{-c}] = \Ex[ \exp(-c \bra{L_N, \log(z - x)})], \quad z \in \C \setminus \R.
\end{equation}
Note that $\bra{\nu_N, x} = a_1 = a_1^{(N)}$. Here we add the super-script to indicate the dependence on $N$. Then, we get the following relation for all three classical beta ensembles
\begin{equation}\label{MKR-classical2}
\Ex[(z - a_1^{(N)})^{-c}] = \Ex[ \exp(-c \bra{L_N, \log(z - x)})], \quad z \in \C \setminus \R.
\end{equation}
From random matrix models $J_N$, we see that $a_1^{(N)}$ converges in distribution to the standard Gaussian distribution, the gamma distribution and the beta distribution, respectively. Thus, taking the limit as $N$ tends to infinity, we obtain the following result on the MKT of the limiting measures in a high temperature regime. 
\begin{theorem}\label{thm:MKT-beta-ensembles}
Denote by $\mu_c^{(G)}$, $\mu_c^{(L)}$ and $\mu_c^{(J)}$ the limiting measure in the high temperature regime $\beta = 2c/N$ of the empirical distribution of Gaussian beta ensembles, beta Laguerre ensembles and beta Jacobi ensembles~\eqref{joint-density}, respectively. Here $c\in (0, \infty)$ is given. Then the following hold.
\begin{itemize}
	\item[\rm(i)] The standard Gaussian distribution $\Normal(0,1)$ and $\mu_c^{(G)}$ are linked by the MKR with parameter $c$.
	\item[\rm(ii)] The gamma distribution $\Gam(\alpha + c, 1)$ and $\mu_c^{(L)}$ are linked by the MKR with parameter $c$.

	\item[\rm(iii)] The beta distribution $\Beta(a+c+1, b+c+1)$ and $\mu_c^{(J)}$ are linked by the MKR with parameter $c$.

\end{itemize}
\end{theorem}

We now turn to investigate eigenvalue processes related to the three beta ensembles which are called beta Dyson's Brownian motions, beta Laguerre processes and beta Jacobi processes. For those eigenvalue processes at high temperature, the mean-field limit, or the limiting behavior of the empirical measure process, has been studied \cite{NTT-2023, Trinh-Trinh-Jacobi,Trinh-Trinh-BLP}. It has been shown by a moment method that moment processes of the limiting probability measure-valued process $\mu_t$ are defined recursively by certain ordinary differential equations. The main goal of this paper is to identify the limit $\mu_t$ by using the MKT. In each of the three cases, we propose a 1d stochastic process whose marginal distribution at time $t$ is the MKT of $\mu_t$ with parameter $c$. 

Let us introduce the result in the Gaussian case  in more detail. The 1d stochastic process here is given by $(\xi + b_t)$, where $\xi$ is a random variable with distribution $\cM_{c}(\mu_0)$ and $b_t$ is a standard Brownian motion independent of $\xi$. Let us relate the result to the case of fixed temperature. When $\beta$ is fixed, the empirical measure process of beta Dyson's Brownian motions converges to a deterministic process of probability measures which can be written as the free convolution of the initial measure and a scaling of the standard semi-circle distribution (see \cite[\S4.3 and \S4.6]{Anderson-book}). In a high temperature regime, the limiting process can be characterized by an integro-differential equation \cite{Cepa-Lepingle-1997}. The construction of 1d stochastic process in this work is equivalent to the statement that the limiting process is the $c$-convolution of the initial measure and a scaling of $\mu_c^{(G)}$ (Theorem~\ref{thm:G-MKT}). Here the concept of $c$-convolution will be defined in Definition~\ref{defn:c-convolution} based on the MKT which continuously interpolates the classical convolution ($c \to 0$) and the free convolution ($c \to \infty$) \cite{Mergny-Potters-2022}.

Motivated by a recent result on  the local scaling of beta Jacobi processes when $\beta$ is fixed \cite{AVW2024}, we also investigate the local scaling in the Laguerre case and the Jacobi case. Roughly speaking, under suitable scaling, the limiting process in the Gaussian case can be found in both the Laguerre case and the Jacobi case. The limiting process in the Laguerre case can be found in the Jacobi case as well.  

The paper is organized as follows. In Section~\ref{sect:classical}, we give the proof of Theorem~\ref{thm:MKT-beta-ensembles}. Section~\ref{sect:MKT} introduces several properties of the MKT needed for the rest of the paper. The limiting behavior of three eigenvalue processes will be studied sequentially in Section~\ref{sect:G}, Section~\ref{sect:L} and Section~\ref{sect:J}.

\section{Classical beta ensembles at high temperature}\label{sect:classical}
We give the proof of three parts in Theorem~\ref{thm:MKT-beta-ensembles} in the next three subsections. As mentioned in the introduction, Theorem~\ref{thm:MKT-beta-ensembles} follows by letting $N \to \infty$ in the relation~\eqref{MKR-classical2}. To do that, we need the following result.

\begin{lemma}\label{lem:log}
Assume that as $N \to \infty$,
\[
	\bra{L_N, f} \to \bra{\mu_c, f} \quad \text{almost surely},
\]
for any continuous function $f$ of polynomial growth. Then for $z \in \C \setminus \R$,
\[
	\Ex[ \exp(-c \bra{L_N, \log(z - x)})] \to  \exp\left(-c  \bra{\mu_c, \log(z - x)} \right)\quad \text{as} \quad N \to \infty.
\]
\end{lemma}
\begin{remark}
	We note here that the assumption in Lemma~\ref{lem:log} holds for all three classical beta ensembles at high temperature \cite{B-GP15, Nakano-Trinh-2018, Trinh-Trinh-Jacobi, Trinh-Trinh-2021}.
\end{remark}
\begin{proof}
	For $z \in \C \setminus \R$, note that the function $\log(z-x)$ is continuous of polynomial growth. Then by the assumption, 
\[
	\bra{L_N, \log(z - x)} \to \bra{\mu_c, \log(z - x)}, \quad \text{almost surely as $N \to \infty$}.
\]
It now follows from the continuous mapping theorem that 
\[
	 \exp(-c \bra{L_N, \log(z - x)}) \to \exp(-c\bra{\mu_c, \log(z - x)}), \quad \text{almost surely as $N \to \infty$}.
\]
The mean value also converges because
\[
	 \exp(-c \bra{L_N, \log(z - x)}) = \prod_{i=1}^N (z - \lambda_i )^{-\frac cN}
\]
is bounded. The proof is complete.
\end{proof}

%\begin{theorem}\label{thm:MKT-beta-ensembles}
%For $z \in \C \setminus \R$, as $N \to \infty$ with $\beta = 2c/N$, 
%\[
%	\Ex[(z - \bra{\mu_N, x})^{-c}] \to \exp(-c \bra{\mu_c, \log(z - x)}).
%\]
%\end{theorem}
%\begin{proof}
%For $z \in \C \setminus \R$, the real part of $\log(z - x)$ is continuous, bounded below and of polynomial growth. The image part of $\log(z - x)$ is bounded continuous function. Thus, 
%\[
%	\bra{L_N, \log(z - x)} \to \bra{\mu_c, \log(z - x)} \quad \almostsurely \as N \to \infty.
%\] 
%It then follows that 
%\[
%	\Ex[\exp(-c \bra{L_N, \log(z - x)})] \to \exp(-c \bra{\mu_c, \log(z - x)}).
%\]
%by the bounded convergence theorem. The desired result immediately follows.
%\end{proof}

\subsection{Gaussian beta ensembles}
The random matrix model for Gaussian beta ensembles was introduced in \cite{DE02}. Let
\[
	J_N^{(G)} = \begin{pmatrix}
		a_1	&b_1	\\
		b_1	&a_2		&b_2\\
		&	\ddots	&\ddots	&\ddots\\
			&&b_{N-1}	&a_N
	\end{pmatrix}
	\sim
	\begin{pmatrix}
		\Normal(0,1)	&\chitilde_{(N-1)\beta}	\\
		\chitilde_{(N-1)\beta}	&\Normal(0,1)		&\chitilde_{(N-2)\beta}\\
		&	\ddots	&\ddots	&\ddots\\
			&&\chitilde_{\beta}	&\Normal(0,1)
	\end{pmatrix}.
\]
Here $\chitilde_{k} = \frac1{\sqrt 2} \chi_k$ denotes the chi distribution with $k$ degrees of freedom, divided by $\sqrt 2$.
To be more precise, the random variables $a_1, \dots, a_N, b_1, \dots, b_{N-1}$ are independent with $a_i \sim \Normal(0,1), i = 1, \dots, N$ and $b_i \sim \chitilde_{(N-i)\beta}, i = 1, \dots, N-1$.
Then the eigenvalues $\lambda_1, \dots, \lambda_N$ of $J_N^{(G)}$ follow the Gaussian beta ensemble, that is,  
\begin{equation}\label{GbE}
	(\lambda_1, \dots, \lambda_N) \propto \prod_{i<j}|\lambda_j - \lambda_i|^{\beta} \exp\left(-\frac12 (\lambda_1^2 + \cdots + \lambda_N^2) \right).
\end{equation}

Since the distribution of $a_1\sim \Normal(0,1)$ does not change as $N$ tends to infinity, taking the limit in the relation~\eqref{MKR-classical2}, we get the desired result: 
\[
	\Ex[(z - \Normal(0,1))^{-c}] = \exp\left(-c \int_\R \log(z - u)d\mu_c^{(G)}(u) \right) , \quad z \in \C \setminus \R.
\]
Theorem~\ref{thm:MKT-beta-ensembles}(i) is proved. We remark here that another proof of this part can be found in \cite{BGCG2022, Mergny-Potters-2022}.

\subsection{Beta Laguerre ensembles}
The random matrix model for beta Laguerre ensembles was introduced in \cite{DE02}.
Let $J_N^{(L)} := B_N (B_N)^\top$ be a Jacobi matrix, where $B_N$ is a bidiagonal matrix consisting of independent random variables with distributions
\[
	B_N = \begin{pmatrix}
		\chitilde_{2\alpha + \beta(N-1)}	\\
		\chitilde_{\beta (N-1)}	&\chitilde_{2\alpha + \beta(N-2)}		\\
		&\ddots	&\ddots	\\
		&&\chitilde_{\beta}	&\chitilde_{2\alpha}
	\end{pmatrix}.
\]
Here $(B_N)^\top$ denotes the transpose of $B_N$.
Recall that $\alpha, \beta  > 0$.
Then the joint density of the eigenvalues of $J_N^{(L)}$ is proportional to 
\begin{equation}\label{bLE}
	 \prod_{i<j}|\lambda_j - \lambda_i|^{\beta}\prod_{l = 1}^N \left(\lambda_l^{\alpha-1} e^{- \lambda_l} \right),\quad  (\lambda_i > 0).
\end{equation}
In this case, as $N \to \infty$ with $\beta = 2c/N$,
\[
	a_1^{(N)}\sim \chitilde_{2\alpha + \beta(N-1)}^2 = \Gam\left(\alpha + c - \frac cN, 1\right) \dto \Gam(\alpha + c, 1),
\]
where `$\dto$' denotes the convergence in distribution. Letting $N\to\infty$ in equation~\eqref{MKR-classical2}, we get 
\[
	\Ex[(z -\Gam(\alpha + c, 1))^{-c}] = \exp\left(-c \int_{(0, \infty)} \log(z - u)d\mu_c^{(L)}(u) \right) , \quad z \in \C \setminus \R.
\]
Here Lemma~\ref{lem:log} has been used.
In other words, the gamma distribution $\Gam(\alpha + c, 1)$ and the limiting measure $\mu_c^{(L)}$ are linked by the MKR with parameter $c$. Theorem~\ref{thm:MKT-beta-ensembles}(ii) is proved.

\subsection{Beta Jacobi ensembles}
Let $p_1, \dots, p_N$ and $q_1, \dots, q_{N - 1}$ be independent random variables having beta distributions with parameters
\begin{align*}
	p_n &\sim \Beta((N - n) \tfrac{\beta}2 + a + 1, (N - n) \tfrac{\beta}2 + b + 1),\\
	q_n &\sim \Beta((N - n) \tfrac{\beta}2, (N - n - 1) \tfrac{\beta}2 + a + b + 2),
\end{align*}
where $a, b > -1$ and $\beta > 0$. 
Define
\begin{align*}
	s_n &= {p_n(1 - q_{n - 1})},\quad n = 1, \dots, N, \quad (q_0 = 0),\\
	t_n &= {q_n(1 - p_n)}, \quad n = 1, \dots, N - 1,
\end{align*}
and form a Jacobi matrix
\[
	J_{N}^{(J)} = \begin{pmatrix}
		\sqrt{s_1}	\\
		\sqrt{t_1}	&\sqrt{s_2}		\\
		&\ddots	&\ddots \\
		&& \sqrt{t_{N - 1}}	& \sqrt{s_N}
	\end{pmatrix}
	 \begin{pmatrix}
		\sqrt{s_1} 	&\sqrt{t_1}	\\
			&\sqrt{s_2}		&\sqrt{t_2}		\\
		&&\ddots	&\ddots \\
		&&& \sqrt{s_N}
	\end{pmatrix}.
\]
Then the eigenvalues $\lambda_1, \lambda_2, \dots, \lambda_N$ of $J_{N}^{(J)}$ are distributed according to the following beta Jacobi ensemble \cite{Killip-Nenciu-2004}
\begin{equation}\label{bJE}
(\lambda_1, \dots, \lambda_N) \propto   \prod_{i < j} |\lambda_j - \lambda_i|^\beta \prod_{l = 1}^N \lambda_l^a (1 - \lambda_l)^b, \quad (\lambda_i \in (0,1)).
\end{equation}

In this case, 
\[
	a_1^{(N)}   \sim \Beta((N - n) \tfrac{\beta}2 + a + 1, (N - n) \tfrac{\beta}2 + b + 1)  \dto \Beta(c+a+1, c+b+1).
\]
The third statement in Theorem~\ref{thm:MKT-beta-ensembles} follows, which completes the proof.\qed

\section{Moment relations in the MKR}\label{sect:MKT}

Denote the moments of $\nu$ and $\mu$ by 
\[
	h_m = \int_\R t^m d\nu(t), \quad p_m = \int_\R t^m d\mu(t),
\]
in case the two probability measures have all finite moments. When they are compactly supported, by expanding in power series both sides of the Markov--Krein relation, we obtain the following moment relation \cite{Faraut-Fourati-2016}.
\begin{proposition}\label{prop:MKT}
	Two probability measures $\nu$ and $\mu$ with compact support are linked by the Markov--Krein relation with parameter $c > 0$, if and only if the moments satisfy
\[
	\sum_{m=0}^\infty \frac{(c)_m}{m!} h_m w^m = \exp \left(c	\sum_{m=1}^\infty \frac{p_m}{m}w^m	\right),
\]
for sufficiently small $w$. It follows that $h_m$ can be written as a polynomial in $p_1, \dots, p_m$:
\begin{equation}\label{moment-relation}
	h_m = \frac{m!}{(c)_m} \sum_{k=1}^m \frac {c^k}{k!}	\sum_{\alpha_i \ge 1; \alpha_1 + \cdots + \alpha_k = m} \frac{p_{\alpha_1}}{\alpha_1} \cdots \frac{p_{\alpha_k}}{\alpha_k}.
\end{equation}
Here $(c)_m $  is the Pochhammer symbol 
\[
	(c)_m = \begin{cases}
		c(c + 1) \cdots (c + m - 1), &m \ge 1,\\
		1	& m=0.
	\end{cases}
\]
\end{proposition}

We also have the following recurrence relation for moments.
\begin{lemma}\label{lem:moment-MKT}
The moment relation~\eqref{moment-relation} is equivalent to the following  recurrence relation
\begin{equation}\label{moment-MKR}
	(c)_n h_n = c \sum_{i=0}^{n-1}  \frac{(n-1)!}{i!} (c)_i h_i p_{n-i}, \quad n = 1, 2,\dots, \quad h_0 = 1.
\end{equation}

\end{lemma}
\begin{proof}
We use an idea from \cite{Smith-1995}.
Let 
\[
	\phi(w) = \sum_{m=0}^\infty \frac{(c)_m}{m!} h_m w^m, \quad \psi(w)  = c	\sum_{m=1}^\infty \frac{p_m}{m}w^m.
\]
Then  
\(
	\phi(w) = \exp(\psi(w)),
\)
for sufficiently small $w$. 
By taking the derivative of both sides, we get that ${\rm D}(\phi(w)) = \phi (w) {\rm D}(\psi(w))$. Now by using a binomial formula for the $(n-1)$st derivative of the product of two functions, we arrive at
\[
	{\rm D}^n(\phi(w)) = \sum_{i=0}^{n-1} \binom{n-1}{i} {\rm D}^{n-i}(\psi(w)) {\rm D}^i(\phi(w)).
\]
This identity at $w=0$ gives us the desired relation. The proof is complete.
\end{proof}

\begin{theorem}\label{thm:MKT-moment}
	Let $\mu$ be a probability measure on $\R$ with all finite moments
\[
	p_m = \int u^m d\mu(u) < \infty, \quad m = 0, 1,2,\dots.
\]
For $c > 0$,
define the sequence $\{h_n\}_{n \ge 0}$ by the moment relation~\eqref{moment-MKR}.
Then $\nu = \cM_{c}(\mu)$ has moments $\{h_n\}$. Consequently, when two probability measures $\nu$ and $\mu$ are determined by moments, the above moment relation holds, if and only if $\nu = \cM_{c}(\mu)$.
\end{theorem}
\begin{proof}
We claim that there is a sequence of probability measures 
\[
	\mu_N = \frac{1}{N} \sum_{i=1}^N \delta_{a_i}
\]
converging weakly to $\mu$ such that the convergence of moments holds, that is, for any $n$
\[
	\int x^n d\mu_N(x) \to \int x^n d\mu(x) \quad \text{as}\quad N \to \infty.
\]
Indeed, let $\{\xi_k\}$ be an i.i.d.\ sequence  of random variables defined on some probability space $(\Omega, \cF, \Prob)$ with the common distribution $\mu$. Then for almost surely $\omega \in \Omega$, the empirical distribution 
\begin{equation}\label{empirical}
	\frac 1N \sum_{i=1}^N \delta_{\xi_i(\omega)} \text{ converges weakly to $\mu$.}
\end{equation}
Moreover, for each $n$, since $\Ex[(\xi_k)^n] =  \int x^n d\mu(x) < \infty$, by the strong law of large numbers, almost surely,
\begin{equation}\label{moments}
	\frac{(\xi_1(\omega))^n + \cdots + (\xi_N(\omega))^n}{N}  \to \int x^n d\mu(x) \quad \text{as} \quad N \to \infty.
\end{equation}
Take an $\omega \in \Omega$ such that both \eqref{empirical} and \eqref{moments} hold, and set $a_i = \xi_i(\omega)$, we get the desired sequence $\mu_N = N^{-1} \sum_{i=1}^N \delta_{a_i}$.

Let $\nu_N = \cM_{c}(\mu_N)$ be the MKT of $\mu_N$. Then using the moment relation in Lemma~\ref{lem:moment-MKT}, we deduce by induction that for any $n = 1,2,\dots,$
\begin{equation}\label{moment-nun}
	\lim_{N \to \infty} \int_\R x^n d\nu_N(x) = h_n.
\end{equation}
It follows that the sequence $\{\nu_N\}$ is tight. Let $\nu$ be a weak limit of some subsequence, that is,
\[
	\nu_{N_k} \wto \nu \quad \as \quad k  \to \infty,
\] 
for some subsequence $\{N_k\}_k$.

For $z \in \C \setminus \R$, since the function $(z - t)^{-c}$ is a (complex-valued) bounded continuous function, 
\[
	\int_\R \frac{1}{(z - t)^c} d\nu_{N_k}(t) \to \int_\R \frac{1}{(z - t)^c} d\nu(t) \quad \as \quad k \to \infty.
\]
Now, the convergence of moments is equivalent to the statement that for any polynomial $p$, 
\[
	\bra{\mu_N, p} \to \bra{\mu, p} \quad \text{as} \quad N \to \infty,
\]
which implies that for any continuous function $f$ of polynomial growth \cite[Lemma~2.1]{Trinh-ojm-2018}, 
\[
	\bra{\mu_N, f} \to \bra{\mu, f} \quad \text{as} \quad N \to \infty.
\]
Consequently, 
\[
	\int_\R \log(z - u) d\mu_N(u) \to \int_\R \log(z - u) d\mu(u) \quad \as \quad N \to \infty.
\]
Therefore, the probability measure $\nu$ and $\mu$ are linked by the MKR with parameter $c$, or $\nu = \cM_{c}(\mu)$.

Finally, it follows from the condition~\eqref{moment-nun} that each moment of $\nu_{N_k}$ converges to the corresponding moment of $\nu$, implying moments of $\nu$ coincide with the sequence $\{h_n\}$. The proof is complete.
\end{proof}

\begin{remark}
From the construction in the proof of Theorem~\ref{thm:MKT-moment}, we observe that when $\supp \mu  \subset [a, b]$, then $\supp \cM_{c}(\mu) \subset [a,b]$. In particular, when $\mu$ is a probability measure on $[0, \infty)$ (resp.\ on $[0,1]$), the MKT $ \cM_{c}(\mu)$ is also a probability on $[0, \infty)$ (resp.\ on $[0,1]$).
\end{remark}

\begin{lemma}\label{lem:phi-psi}
Let $\{h_n\}_{n \ge 0}$ and $\{p_n\}_{n \ge 0}$ be two sequences with $h_0 = p_0 = 1$. For $c>0$, let 
\[
	\phi(z) = \sum_{n=0}^\infty \frac{(c)_n}{n!} h_n \frac{1}{z^{c+n}},  \quad
	\psi(z) = \sum_{n=0}^\infty{p_n} \frac{1}{z^{n+1}}
\]
be formal power series. Then the relation~\eqref{moment-MKR} is equivalent to the following relation of formal power series
\[
	\psi = -\frac1c \frac{\phi_z}{\phi}, \quad \text{or} \quad \phi_z = -c \phi \psi.
\]
Here $\phi_z$ denotes the formal derivative with respect to $z$,
\[
	\phi_z(z) = \sum_{n=0}^\infty \frac{-(c)_{n+1}}{n!} h_n \frac{1}{z^{c+n+1}}.
\]
\end{lemma}
\begin{proof}
We omit detailed arguments which can be done by a straightforward calculation.
\end{proof}

Note that when $\{h_n\}_{n \ge 0}$ and $\{p_n\}_{n \ge 0}$ are moments of two probability measures $\nu$ and $\mu$ with compact support, respectively, the two formal power series are convergent, provided that $z \in \C$ is large enough, and we get
\[
	\phi(z) = \int \frac{d\nu(x)}{(z - x)^c}, \quad \psi(z) = \int \frac{d\mu(u)}{z - u}, \quad \text{for $z$ large enough}.
\]

\begin{definition}\label{defn:c-convolution}
	The probability measure $\mu$ is the $c$-convolution of two probability measures $\mu_1$ and $\mu_2$ if 
\[
	\cM_{c}(\mu) = \cM_{c}(\mu_1) * \cM_{c}(\mu_2).
\]
Here `$*$' denotes the (classical) convolution of two probability measures.
We write $\mu = \mu_1 \boxplus_c \mu_2$. Note that we have implicitly assumed that all the MKTs exist.
\end{definition}

\begin{remark}
The $c$-convolution interpolates classical convolution ($c \to 0$) and free convolution ($c \to \infty$) \cite{BGCG2022, Mergny-Potters-2022}.
We recall here an important open question mentioned in \cite{BGCG2022, Mergny-Potters-2022}: whether the $c$-convolution of two probability measures is again a probability measure? In other words, for two probability measures $\mu_1$ and $\mu_2$, is $\cM_{c}(\mu_1) * \cM_{c}(\mu_2)$ the MKT of a probability measure?
\end{remark}

Let $\cP_G$ be a set of probability measures on $\R$ with all finite moments such that for $\mu \in \cP_G$, there is a $\Lambda > 0$ such that 
\[
	\left| \int_\R x^n d\mu(x) \right| \le (\Lambda n)^n, \quad n =0,1,2,\dots.
\]
We recall well-known Carleman's condition which states that the condition 
\begin{equation}
	\sum_{n=1}^\infty m_{2n}^{-\frac1{2n}} = +\infty
\end{equation}
is sufficient for $\mu$ to be determined by moments $\{m_n\}_{n \ge 0}$. Thus, probability measures in $\cP_G$ are determined by moments. We claim that $\cP_G$ is closed under the MKT, that is, for $\mu \in \cP_G$, the MKT $\cM_{c}(\mu) \in \cP_G$. It is a consequence of the following result.
\begin{lemma}\label{lem:moment-bound}
Let $c>0$ be given.
Assume that moments $\{h_n\}_{n \ge 0}$ and $\{p_n\}_{n \ge 0}$ satisfy the relation~\eqref{moment-MKR}.
Assume that $|p_n| \le (\Lambda n)^n, n = 0,1,2,\dots$, for some $\Lambda > 0$. Then there is an $M>0$ such that 
\[
	|h_n| \le (M n)^n, \quad n = 0,1,2,\dots.
\]
\end{lemma}
\begin{proof}
	Let us first consider the case $c \ge 1$. If follows from the following estimate
	\[
		c   \frac{(n-1)!}{i!} \frac{(c)_i}{(c)_n} = 
			c\frac{(i+1)\cdots (n-1)}{(c+i)\cdots (c+n-1)} \le \frac cn,
	\]
that 
\[
	|h_n| \le \frac cn \sum_{i=0}^{n-1} |h_i| |p_{n-i}|.
\]
From which, by induction we get $|h_n| \le (c\Lambda n)^n$, implying the desired estimate with $M = c \Lambda$.

Next we consider the case $c \in (0,1)$. In this case, we have
	\[
		c   \frac{(n-1)!}{i!} \frac{(c)_i}{(c)_n} = 
			c\frac{(i+1)\cdots (n-1)}{(c+i)\cdots (c+n-1)} \le 1,
	\]
which implies 
\[
	|h_n| \le \sum_{i=0}^{n-1} |h_i| |p_{n-i}|.
\]
Then the desired estimate with $M=2\Lambda$ can be proved by induction. The proof is complete.
\end{proof}

For the Laguerre case, we consider probability measures on $[0, \infty)$. Carleman's condition in this case is as follows: a probability measure $\mu$ on $[0, \infty)$ is uniquely determined by moments $\{m_n\}$, if 
\[
	\sum_{n=1}^\infty m_n^{-\frac1{2n}} = \infty.
\]
Now let $\cP_L$ be the set  of probability measures $\mu$ on $[0, \infty)$ with moments $\{m_n\}_{n \ge 0}$  such that 
\[
	m_n \le (\Lambda n)^{2n}, \quad n = 0,1,2,\dots,
\]
for some $\Lambda > 0$. Then probability measures in $\cP_L$ are determined by moments and $\cP_L$ is also closed under the MKT.

\section{Beta Dyson's Brownian motions}\label{sect:G}
This section deals with beta Dyson's Brownian motions which are related to Gaussian beta ensembles.
They are defined to be the strong solution of the following system of stochastic differential equations (SDEs)
\begin{equation}
	d\lambda_i(t)= d  b_i(t)  + \dfrac\beta2 \sum\limits_{j : j \neq i} \dfrac{1}{\lambda_i(t) - \lambda_j(t)} dt,\quad i = 1, \dots, N,
\end{equation}
with given initial condition under the constraint that $\lambda_1(t) \le \lambda_2(t) \le \cdots \le \lambda_N(t)$ for all $t > 0$ \cite{Cepa-Lepingle-1997}. Here $\{b_i(t)\}_{i=1}^N$ are independent standard Brownian motions. Under the trivial initial condition, that is, $\lambda_i(0) = 0, i = 1, \dots, N$, it is known that for any $t > 0$, the joint distribution of $\{\lambda_i(t)/\sqrt t\}_{i=1}^N$ is the ordered Gaussian beta ensemble
\begin{equation*}
	(\lambda_1, \dots, \lambda_N) \propto \prod_{i<j}|\lambda_j - \lambda_i|^{\beta} \exp\left(-\frac12 (\lambda_1^2 + \cdots + \lambda_N^2) \right), \quad \lambda_1 \le \lambda_2 \le \cdots \le \lambda_N.
\end{equation*}

In a high temperature regime where $\beta = 2c/N$, the mean-field limit has been studied in \cite{Cepa-Lepingle-1997} by a standard method. It has been also studied in \cite{NTT-2023} by a moment method under the trivial initial condition. We reproduce the moment method here to study the limiting behavior of the scaled processes $x_i(t) = \sigma \lambda_i(t)$, for fixed $\sigma > 0$. The SDEs for $\{x_i(t)\}$ are given by  
\begin{equation}\label{SDE-G-sigma}
	dx_i(t)= \sigma d  b_i(t)  + \sigma^2\dfrac{c}N \sum\limits_{j : j \neq i} \dfrac{1}{x_i(t) - x_j(t)} dt, \quad i = 1,2,\dots, N.
\end{equation}

From now on, we denote the empirical measure process of $\{x_i(t)\}_{i=1}^N$ by
\begin{equation}
	\mu_t^{(N)} = \frac1N \sum_{i=1}^N \delta_{x_i(t)},
\end{equation}
and its moment processes by
\begin{equation}
	S_n^{(N)}(t) = \bra{\mu_t^{(N)}, x^n} = \frac1N \sum_{i=1}^N x_i(t)^n,\quad n = 0,1,\dots.
\end{equation}

For $f \in C^2(\R)$,  by It\^o's formula \cite{Cepa-Lepingle-1997, NTT-2023},
\begin{align}
	d\bra{\mu^{(N)}_t, f} 
	&=\sigma \frac1N\sum_{i=1}^N f'(x_i(t))  d b_i(t)  + \sigma^2 \frac c2 \iint \frac{f'(x) - f'(y)}{x - y}d\mu_t^{(N)}(x) d\mu_t^{(N)}(y)dt \notag\\
	&\quad + \sigma^2 \bra{\mu_t^{(N)}, \frac12 f''} dt  - \sigma^2 \frac{c}{2N} \bra{\mu_t^{(N)}, f''}dt. \label{Ito-f}
\end{align}
Note that the above formula holds when $x_1(t), \dots, x_N(t)$ are all distinct, which occurs almost surely for almost every $t \in \R$ \cite{Cepa-Lepingle-1997}.

With $f = x^n$, the formula~\eqref{Ito-f} implies a recurrence relation for moment processes
\begin{align*}
	dS_n^{(N)}(t) &=  	\sigma\frac{ n}N\sum_{i=1}^N x_i(t)^{n - 1}  d b_i(t) + \sigma^2\frac{c n}2 \sum_{j = 0}^{n-2} S_{j}^{(N)}(t) S_{n-2-j}^{(N)}(t) dt \\
	&\quad + \sigma^2\frac12 n(n-1) S_{n-2}^{(N)}(t) dt - \sigma^2\frac c{2N} n(n-1) S_{n-2}^{(N)}(t) dt,
\end{align*}
or in the integral form 
\begin{align}
	S_n^{(N)}(t) &=  S_n^{(N)}(0) + 	\sigma\frac{ n}N\sum_{i=1}^N \int_0^t x_i(u)^{n - 1}  d b_i(u) + \sigma^2\frac{c n}2 \sum_{j = 0}^{n-2} \int_0^t  S_{j}^{(N)}(u) S_{n-2-j}^{(N)}(u) du \notag \\
	&\quad +\sigma^2 \frac12 n(n-1) \int_0^t S_{n-2}^{(N)}(u) du - \sigma^2 \frac c{2N} n(n-1) \int_0^t S_{n-2}^{(N)}(u) du, \quad (n \ge 1).	\label{Gauss-Sn}
\end{align}

Consider the space $C([0, T])$ of continuous functions on $[0,T]$ endowed with the supremum norm $\norm{\cdot}_\infty$. Then moment processes $S^{(N)}_n$ are $C([0, T])$-valued random elements. By induction, we will show the convergence of $S^{(N)}_n$ to a deterministic limit as $N \to \infty$. 
\begin{definition}
A sequence $\{X^{(N)}\}_N$ of $C([0, T])$-valued random elements is said to converge in probability to a non-random element $x \in C([0, T])$ if for any $\varepsilon > 0$, 
\[
		\lim_{N \to \infty} \Prob(\norm{X^{(N)} - x}_\infty \ge \varepsilon) =0.
\] 
%We denote the convergence in probability by `$\Pto$'.
\end{definition}

\begin{theorem}\label{thm:G-moments}
Assume that each moment of the initial empirical measure $\mu_0^{(N)}$ converges to a limiting moment, that is, for each $n$,
\begin{equation}\label{moment-assumption}
	S_n^{(N)}(0) = \frac{1}{N} \sum_{i=1}^N x_i(0)^n \to a_n \quad \text{as} \quad N \to \infty.
\end{equation}
Then it holds that 
\[
	S_n^{(N)}(t) \to m_n(t) \quad \text{as} \quad N \to \infty,
\]
in probability as random elements in $C([0, T])$. Here the limiting moment processes are recursively defined by $m_0(t) = 1, m_1(t) = a_1,$
\begin{align}
	m_n(t) &= a_n + \sigma^2 \frac{cn}2 \sum_{j=0}^{n-2}\int_0^t m_j(u) m_{n-2-j}(u) du\notag\\
	&\quad +\sigma^2 \frac12 n(n-1) \int_0^t m_{n-2}(u) du, \quad (n \ge 2).\label{Gauussian-recurrence-moment}
\end{align}
\end{theorem}
\begin{proof}
	The proof is straightforward by induction with the help of Lemma~\ref{lem:Gaussian-MTG}.
\end{proof}

\begin{lemma}\label{lem:Gaussian-MTG}
	Under the moment assumption~\eqref{moment-assumption}, the martingale part
\[
	M_n^{(N)}(t) = \sigma\frac{ n}N\sum_{i=1}^N \int_0^t x_i(u)^{n - 1}  d b_i(u) 
\]
converges to $0$ in probability as random elements in $C([0, T])$.
\end{lemma}
\begin{proof}
The arguments here are similar to those used in proving Lemma 3.4 in \cite{Trinh-Trinh-BLP}. For the reader's convenience, we sketch key ideas. 

Step 1. By using Doob's martingale inequality, for any given $\varepsilon > 0$, it holds that 
\[
	\Prob\left( \sup_{0 \le t \le T} |M_n^{(N)}(t)| \ge \varepsilon \right) \le \frac{1}{\varepsilon^2} \Ex[M_n^{(N)}(T)^2].
\]
Thus, it suffices to show that $\Ex[M_n^{(N)}(T)^2] \to 0$ as $N \to \infty$.

Step 2. The mean value $\Ex[M_n^{(N)}(T)^2] $ can be expressed as
\[
	\Ex[M_n^{(N)}(T)^2] = \frac{ \sigma^2 n^2}{N} \Ex \left [\int_0^T \frac{\sum_{i=1}^N x_i(u)^{2n-2}}{N}\right ] =  \frac{ \sigma^2 n^2}{N} \int_0^T \Ex[S_{2n-2}^{(N)}(t)] dt.
\] 

Step 3. Under the moment assumption~\eqref{moment-assumption}, we show that for even $n$, 
\begin{equation}\label{uniform-bdn}
	\Ex[S_n^{(N)}(t)] \le C_n, \quad (t \in [0, T], N \ge 2),
\end{equation}
for some constant $C_n$ depending only on $n$. The proof is complete by combining the three steps. Now, it remains to show the estimate~\eqref{uniform-bdn}. Denote by $m_n^{(N)}(t) = \Ex[S_n^{(N)}(t)]$. When $n$ is even, we use the inequality
\begin{align*}
	S_{j}^{(N)}(u) S_{n-2-j}^{(N)}(u) &= \frac1{N^2} \left(\sum_{i=1}^N x_i(u)^j \right)\left(\sum_{i=1}^N x_i(u)^{n-2-j} \right) \\
	&\le  \frac1{N^2} \left(\sum_{i=1}^N |x_i(u)|^j \right)\left(\sum_{i=1}^N |x_i(u)|^{n-2-j} \right)\\
	&\le \frac1{N} \left(\sum_{i=1}^N |x_i(u)|^{n-2} \right) = S_{n-2}^{(N)}(u)
\end{align*}
to deduce that 
\[
	0 \le m_n^{(N)}(t) \le a_n + \sigma^2(c+1)\frac{n(n-1)}2 \int_0^t m_{n-2}^{(N)}(u)du.
\]
Then, the conclusion follows by induction.
\end{proof}

\begin{theorem}\label{thm:G-measures}
	Assume the moment condition~\eqref{moment-assumption}. Assume further that the limiting moments $\{a_n\}$ satisfy
\begin{equation}\label{Gaussian-Carleman-initial}
	|a_n| \le (\Lambda n)^n, \quad \text{for all $n \ge 1$}, 
\end{equation}
where $\Lambda > 0$ is a constant. Then for any $t \ge 0$, there is a unique probability measure $\mu_t$ having moments $\{m_n(t)\}_{n \ge 0}$. Moreover, the probability measure-valued process $\mu_t$ is continuous and the empirical measure process $(\mu_t^{(N)})_{0 \le t \le T}$ converges to $(\mu_t)_{0 \le t \le T}$ in probability under the topology of uniform convergence. 
\end{theorem}
\begin{proof}
We will show in Lemma~\ref{lem:initial-moments} that for each $t$, the sequence $\{m_n(t)\}_{n \ge 0}$ satisfies the same type of estimate as $\{a_n\}$ (equation~\eqref{Gaussian-Carleman-type}). Such estimate clearly implies Carleman's condition, and thus, the limiting moments $\{m_n(t)\}$ uniquely determine a probability measure $\mu_t$. In other words, $\mu_t \in \cP_G$.

Since each moment process $m_n(t)$ is continuous and $\mu_t$ is uniquely determined by moments, it follows that $\mu_t$ is a continuous process of probability measures and the empirical measure process $(\mu_t^{(N)})_{0 \le t \le T}$ converges to $(\mu_t)_{0 \le t \le T}$ in probability under the topology of uniform convergence
(see Appendix A in \cite{Trinh-Trinh-BLP}). The proof is complete.
\end{proof}

\begin{lemma}\label{lem:initial-moments}
Assume that the sequence of limiting moments $\{a_n\}$ satisfies the condition~\eqref{Gaussian-Carleman-initial}. Then for any $T > 0$, there is a constant $K > 0$ such that 
\begin{equation}\label{Gaussian-Carleman-type}
	|m_n(t)| \le (K n)^n, \quad \text{for all $t \in  [0, T]$, and for all $n \ge 0$}.
\end{equation}
\end{lemma}
\begin{proof}
We take a constant $K$ as
\begin{equation}\label{K}
	K = \max\{\sigma^2 (c+1) T, 2\Lambda^2 \}^{1/2}, \quad\text{or} \quad K^2 = \max\{\sigma^2 (c+1) T, 2\Lambda^2 \}.
\end{equation}
(The reason for that will be clear from the argument below.) Since $m_0(t) = 1$ and $m_1(t) = a_1$, it is clear that for $k = 0, 1$,
\[
	|m_k(t)| \le (K k)^k, \quad \text{for all $t \in  [0, T]$}.
\]
We now show the estimate~\eqref{Gaussian-Carleman-type} by induction. Assume that the above estimate holds for any $k = 0,1,\dots, n-1, (n \ge 2)$. We use the definition of $m_n(t)$ to show that it holds for $k = n$. 

Similar to Step 3 in the proof of Lemma~\ref{lem:Gaussian-MTG}, for even $n$, we deduce that 
\begin{align*}
	0 \le m_n(t) &\le a_n +  \sigma^2(c+1)\frac{n(n-1)}2 \int_0^t m_{n-2}(u)du.
\end{align*}
Then using the condition on the initial values and the induction assumption, we get that
\begin{align*}
	0 \le m_n(t) 
	&\le  (\Lambda n)^n + \frac{\sigma^2 (c+1) t}{2} n(n-1) (K(n-2))^{n-2} \\	&\le  (\Lambda n)^n + \frac{\sigma^2 (c+1) T}{2} K^{n-2} n^n \\
	&\le \frac{K^n}2 n^n +  \frac{K^n}2 n^n = (Kn)^n.
\end{align*}
Here in the last line, we have used the condition~\eqref{K}. The proof is complete.
\end{proof}

\subsection{Identifying the limiting process}
For simplicity, let $\sigma = 1$ in this section. Under the assumption of Theorem~\ref{thm:G-measures}, it is clear that $\mu_t \in \cP_G$ for any $t > 0$. When the initial condition is trivial, that is, $\lambda_i(0) = 0, i = 1, \dots, N$, recall that the distribution of $\{\lambda_i(t) /\sqrt t\}_{i=1}^N$ is the Gaussian beta ensemble. Thus, as a consequence of Theorem~\ref{thm:MKT-beta-ensembles}, the Gaussian distribution $\Normal(0, t)$ is the MKT of the limiting measure $\mu_t$. Denote by $\rho_t^{(G)}$ the limiting measure process in this trivial case. The aim of this section is to show that 

\begin{theorem}\label{thm:G-MKT} Under the same assumption as in Theorem~{\rm\ref{thm:G-measures}}, it holds that
\[
	\mu_t = \mu_0 \boxplus_c \rho_t^{(G)}.
\]
\end{theorem}

We formulate the result in a different way. Let $\xi$ be a random variable with distribution $\nu_0 = \cM_{c}(\mu_0)$. Let $b_t$ be a standard Brownian motion independent of $\xi$. Then the above statement is equivalent to the statement that 
\[
	\cL(\xi + b_t) = \cM_{c}(\mu_t).
\]
Here $\cL(X)$ denotes the law/distribution of a random variable $X$.

\subsubsection{A direct approach}
\begin{proof}[Proof of Theorem~\rm{\ref{thm:G-MKT}}]
Let
\[
	\psi(t, z) = \int_\R \frac{d\mu_t(u)}{z - u}, \quad t \ge 0, z \in \C_+
\]
be the Stieltjes transform of the limiting process $\mu_t$. For $z \in \C_+$, consider the complex valued function $f(x) = 1/(z-x)$. Note that $\psi(t, z) = \bra{\mu_t, f}$. By using the formula~\eqref{Ito-f} with that $f$, then taking the limit as $N \to \infty$, we arrive at the following partial differential equation (PDE)
\[
	\psi_t = -c \psi \psi_z + \frac12 \psi_{zz}.
\]
Here $\psi_t, \psi_z$ and $\psi_{zz}$ denote the partial derivative with respect to $t$, the partial derivate with respect to $z$ and the second derivative with respect to $z$, respectively. 
It is a complex Burger equation and it admits a unique solution \cite{Cepa-Lepingle-1997}.

Define 
\[
	u(t, z) = \Ex[(z - \xi - b_t)^{-c}], \quad t \ge 0, z \in \C_+.
\]
Then $u$ satisfies the heat equation
\(
	u_t = \frac12 u_{zz}.
\)
Next, we make the change of variables 
\[
	M(t, z) = -\frac1c (\log u)_z = -\frac1c \frac{u_z}{u}.
\]
We derive the PDE for $M$ as follows.
By a direct calculation, we get that
\[
	M_t = -\frac1c (\log u)_{zt} = -\frac1c \left(\frac {u_t}{u} \right)_z =  -\frac1{2c}\left(\frac {u_{zz}}{u} \right)_z,
\]
\[
	M_z = -\frac1c \left(\frac {u_z}{u} \right)_z = -\frac1c \left(\frac {uu_{zz} - u_z^2}{u^2} \right) = -\frac1c \frac {u_{zz}}{u} + cM^2,
\]
and
\[
	M_{zz} =  -\frac1c \left( \frac {u_{zz}}{u} \right)_z+ c(M^2)_z = 2M_t + 2c M M_z.
\]
Therefore,
\[
	M_t = -c M M_z + \frac12 M_{zz}.
\]

Note that by the construction, $M$ and $\psi$ have the same initial condition at $t = 0$. Since the Burger equation admits a unique solution, it follows that 
\[
	M = \psi, \quad \text{or}\quad \psi = -\frac1c \frac{u_z}{u}.
\]
This relation implies that $\cL(\xi + b_t) = \cM_{c}(\mu_t)$ by Lemma~\ref{lem:derivative} below. The proof of Theorem~\ref{thm:G-MKT} is complete.
\end{proof}

\begin{lemma}\label{lem:derivative}
Let $\nu$ and $\mu$ be two probability measures with
\(
	\int_\R \log(1 + |u|) d\mu(u) < \infty.
\)
Let
\[
	\phi(z) = \int_{\R}\frac{d\nu(x)}{(z - x)^c}, \quad \psi(z) =  \int_{\R}\frac{d\mu(x)}{z - x}, \quad z \in \C_+.
\]
Then the following two statements are equivalent
\begin{itemize}
	\item[\rm(i)] $\nu = \cM_{c}(\mu)$;
	\item[\rm(ii)] $\log(\phi)_z = -c \psi$.
\end{itemize}
\end{lemma}
\begin{proof}
It is clear that (i) implies (ii), because 
\[
	\left[\int_\R \log(z - u) d\mu(u)\right]_z  =  \int_{\R}\frac{d\mu(u)}{z - u} = \psi(z).
\]
We now show that (ii) implies (i). It follows from the assumption $\log(\phi)_z = -c \psi$ that
\[
	-c\int_\R \log(z - u) d\mu(u) = \log \phi(z) + const.
\]
Adding both sides by $c \log z$, we get 
\[	
	-c\int_\R \log \frac{(z - u)}{z} d\mu(u) = \log [\phi(z) z^c]+ const.
\]
By letting $z \to \infty$, it follows that the constant vanishes, implying the MKR. The proof is complete.
\end{proof}

\subsubsection{A formal series approach}
We now give another proof of Theorem~\ref{thm:G-MKT} based on a formal series approach.
Recall that the limiting moment processes are recursively defined by $m_0(t) = 1, m_1(t) = a_1,$
\begin{align*}
	m_n(t) &= a_n +  \frac{cn}2 \sum_{j=0}^{n-2}\int_0^t m_j(u) m_{n-2-j}(u) du\notag\\
	&\quad +  \frac12 n(n-1) \int_0^t m_{n-2}(u) du, \quad (n \ge 2).
\end{align*}
Equivalently, they satisfy the following ordinary differential equations (ODEs)
\begin{align*}
	&\frac {d m_n(t)}{dt} =  \frac{cn}2 \sum_{j=0}^{n-2}  m_j(t) m_{n-2-j}(t) +  \frac12 n(n-1) m_{n-2}(t),\quad n \ge 1,
\end{align*}
with initial value $m_n(0) = a_n.$
Observe that $m_n(t)$ is a polynomial in $t$.

Consider a formal power series in $t$ and $z$,
\[
	\psi(t, z) = \sum_{n=0}^{\infty} {m_n(t)} \frac1{ z^{n+1}}.
\]
This is exactly the series expansion of the Stieltjes transform. However, we consider formal power series to avoid convergence issue. The formal partial derivatives with respect to $t$ and $z$ are defined as usual. Then it is clear that the ODEs for $m_n(t)$ are equivalent to the formal PDE
\begin{equation}\label{G-psi}
	\psi_t = -c \psi \psi_z + \frac12 \psi_{zz}.
\end{equation}

Let $h_n(t) = \Ex[(\xi + b_t)^n]$ be the $n$th moment of $\xi + b_t$. Note that $\xi + b_t$ is a Brownian motion starting at $\xi$. We can use It\^o's formula to derive the ODEs
\[
	\frac{d h_n(t)}{dt} = \frac12 n(n-1) h_{n-2}(t), \quad n \ge 1,
\]
with initial condition $h_n(0) = \Ex[\xi^n]$. Again, $h_n(t)$ is a polynomial in $t$. Consider the formal series 
\[
	\phi(t, z) = \sum_{n=0}^\infty \frac{(c)_n}{n!} h_n(t) \frac{1}{z^{c+n}}.
\]
Then the ODEs for $h_n(t)$ are equivalent to the following formal PDE
\begin{equation}\label{G-PDE-phi}
	\phi_t = \frac12 \phi_{zz}.
\end{equation}

%Note that the moment relation of $\{a_n(t)\}$ and $\{m_n(t)\}$ is equivalent to the relation of formal power series 
%\begin{equation}\label{G-phi-psi}
%	\phi_z = -c \phi \psi.
%\end{equation}

Let $\{\tilde m_n(t)\}$ be the moment process defined by the recurrence relation~\eqref{moment-MKR}, that is,
\begin{equation*}
	(c)_n h_n(t) = c \sum_{i=0}^{n-1}  \frac{(n-1)!}{i!} (c)_i h_i(t) \tilde m_{n-i}(t), \quad n = 1, 2,\dots, \quad (\tilde m_0 = 1).
\end{equation*}
Then $\tilde m_n(t)$ is a polynomial in $t$. Define the formal power series 
\[
	\tilde \psi (t, z) = \sum_{n=0}^{\infty} {\tilde m_n(t)}\frac1{ z^{n+1}}.
\]
It follows from Lemma~\ref{lem:phi-psi} that 
\[
	\tilde \psi = -\frac1c \frac{\phi_z}{\phi}.
\]
By the same calculation as in the previous subsection, we get the formal PDE for $\tilde \psi$,
\[
	\tilde \psi_t = -c \tilde \psi \tilde \psi_z + \frac12 \tilde \psi_{zz}.
\]
This PDE implies ODEs for $\tilde m_n(t)$ which are exactly the same as those for $m_n(t)$. We conclude that $\tilde m_n(t) \equiv m_n(t)$, because they have the same initial condition. Consequently, moments of $(\xi + b_t)$ and those of $\mu_t$ satisfy the relation~\eqref{moment-MKR}. It follows from Theorem~\ref{thm:MKT-moment} that
\[
	\cL(\xi + b_t) = \cM_{c}(\mu_t).
\]
The proof is complete.\qed

\section{Beta Laguerre processes}\label{sect:L}
The object in this section is the so-called beta Laguerre processes $0 \le \lambda_1(t) \le \lambda_2(t) \le \cdots \le \lambda_N(t)$ which satisfy the following system of SDEs 
\begin{equation}\label{bLP}
	d\lambda_i(t)=\sqrt{2\lambda_i(t)}  d  b_i(t)   + \alpha  d  t + \frac\beta2 \sum\limits_{j : j \neq i} \dfrac{2\lambda_i(t)}{\lambda_i(t) - \lambda_j(t)} dt,\quad i = 1, \dots, N.
\end{equation}
Here $\alpha > 1/2, \beta > 0$ and $\{b_i(t)\}_{i=1}^N$ are independent standard Brownian motions. 
When $\beta \ge 1$, the above system of SDEs has a unique strong solution  which never collides \cite{Graczyk-Malecki-2014}. For $\beta \in (0,1)$, the eigenvalue processes are defined to be the the squared of type B radial Dunkl processes \cite{Demni-2007-arxiv}. Under the trivial initial condition $\lambda_i(0) = 0, i = 1, \dots, N$, it is known that for $t > 0$, the joint distribution of $\{\lambda_i(t)/t\}_{i=1}^N$ follows the ordered beta Laguerre ensemble 
\begin{equation*}\label{LbE-ordered}
(\lambda_1,\dots,\lambda_N) \propto	\prod_{i<j}|\lambda_j - \lambda_i|^{\beta} \prod_{l=1}^N \lambda_l^{\alpha-1} e^{-\lambda_l}, \quad (0 \le \lambda_1 \le  \cdots \le \lambda_N).
\end{equation*}

We also consider the scaled processes $x_i(t) = \sigma \lambda_i(t)$, for given $\sigma >0$, which satisfy
\begin{equation}
	dx_i(t) = \sqrt \sigma \sqrt{2x_i(t)}  d  b_i(t)   + \sigma \alpha  d  t + \sigma\frac cN \sum\limits_{j : j \neq i} \dfrac{2x_i(t)}{x_i(t) - x_j(t)} dt, \quad i = 1, \dots, N.
\end{equation}
Here we have replaced $\beta$ by ${2c}/N$ in a high temperature regime.

\subsection{Mean-field limit for beta Laguerre processes}
We use the same notations as in the Gaussian case for the empirical measure process
\(
	\mu_t^{(N)}
\)
and its moment processes
\(
	S_n^{(N)}(t), n = 0,1,2,\dots.
\)
The moment method has been used in \cite{Trinh-Trinh-BLP} to study beta Laguerre processes of Ornstein--Uhlenbeck type. Since our model here is slightly different, we only sketch key ideas.

{Step 1.} For $f \in C^2(\R)$,  by It\^o's formula, we obtain that
\begin{align}
	d\bra{\mu^{(N)}_t, f} 
	&=\sqrt\sigma \frac1N\sum_{i=1}^N f'(x_i)\sqrt{2x_i}  d b_i(t) + \sigma \bra{\mu_t^{(N)}, \alpha f'(x) + x f''(x)} dt \notag\\
	&\quad+ \sigma {c}\iint \frac{x f'(x) - y f'(y)}{x - y} d\mu_t^{(N)}(x)d\mu_t^{(N)}(y)  dt \notag\\
	&\quad - \sigma\frac{c}{N} \bra{\mu_t^{(N)}, xf''(x) + f'(x)} dt. \label{Ito-L}
\end{align}
The above formula holds when $x_1(t), \dots, x_N(t)$ are all distinct, which occurs almost surely for almost every $t \in \R$. (For simplicity, we write $x_i$ instead of $x_i(t)$.)

{Step 2.} With $f = x^n$, the above equation yields a recurrence relation for the $n$th moment process $S_n^{(N)}$,
\begin{align}
	d S_n^{(N)}(t)  &= \sqrt{\sigma} \frac{n}N \sum_{i = 1}^N \sqrt{2x_i} x_i^{n-1} db_i  \notag \\
	&\quad + \sigma n(n + \alpha - 1) S_{n-1}^{(N)}(t) dt + \sigma c n \sum_{i = 0}^{n-1} S_i^{(N)}(t) S_{n-i-1}^{(N)}(t) dt \notag  \\
	&\quad  - \sigma \frac{c n^2}N S_{n-1}^{(N)}(t) dt. \label{L-moment}
\end{align}

Step 3.
Let
\[
	M_n^{(N)}(t) =\sqrt{\sigma}\frac{n}{N} \sum_{i = 1}^N \int_0^t   \sqrt{2x_i} x_i^{n-1} db_i = \frac{\sqrt {2\sigma} n}{N} \sum_{i = 1}^N \int_0^t    x_i^{n-1/2} db_i 
\]
be the martingale part with the quadratic variation 
\begin{equation}\label{quadratic-of-M}
	[M_n^{(N)}](t) = \frac{2\sigma n^2}{N}   \int_0^t \frac{\sum_{i = 1}^N x_i^{2n-1}}{N}ds.
\end{equation}
Under the moment condition~\eqref{moment-assumption}, the martingale part
$
	M_n^{(N)}(t)
$
converges to $0$ in probability as random elements in $C([0, T])$.
%\begin{proof}
%	By the same arguments as used in the proof of Lemma~\ref{lem:Gaussian-MTG}, we only need to show that under the moment condition~\eqref{moment-assumption}, 
%\begin{equation}\label{L-mean-moment-bdd}
%	\Ex[S_{n}^{(N)}(t)] \le C_n, \quad (t \in [0, T], N \ge 1),
%\end{equation}
%for some constant $C_n$ depending only on $n$. Note that all $x_i$'s are non-negative, and thus, for any $j = 0,1,\dots, n-1$,
%\[
%	S_j^{(N)}(t) S_{n-1-j}^{(N)}(t) \le S_{n-1}^{(N)}(t).
%\]
%The relation~\eqref{L-moment} implies that for $n \ge 1$,
%\[
%	0 \le \Ex[S_n^{(N)}(t)] \le a_n + \sigma \left(n(n + \alpha - 1) + cn^2 - \frac{cn^2}{N} \right) \int_0^t \Ex[S_{n-1}^{(N)}(u)]du,
%\]
%from which the condition~\eqref{L-mean-moment-bdd} can be easily deduced by induction. The proof is complete.
%\end{proof}

We arrive at the following result.
\begin{theorem}\label{thm:L}
	Under the moment assumption~\eqref{moment-assumption}, it holds that 
\[
	S_n^{(N)}(t) \to m_n(t) \quad \text{as} \quad N \to \infty,
\]
in probability as random elements in $C([0, T])$. Here the limiting moment processes are recursively defined by $m_0(t) = 1,$ 
\begin{align}
	m_n(t) &= a_n + \sigma n (n+\alpha - 1)\int_0^t m_{n-1}(s)ds + \sigma {cn} \sum_{j=0}^{n-1}\int_0^t m_j(u) m_{n-1-j}(u) du,  \label{Laguerre-recurrence-moment}
\end{align}
for $n \ge 1$.
If, in addition, the initial limiting moments $\{a_n\}$ satisfy a condition that
\[
	a_n \le (\Lambda n)^{2n}, \quad n \ge 0,
\]
for some constant $\Lambda > 0$. Then for any $ t \ge 0$, $\{m_n(t)\}_{n \ge 0}$ are moments of a unique probability measure $\mu_t$ on $[0, \infty)$. Moreover, the probability measure-valued process $\mu_t$ is continuous and the empirical measure process $(\mu_t^{(N)})_{0 \le t \le T}$ converges to $(\mu_t)_{0 \le t \le T}$ in probability under the topology of uniform convergence. 
\end{theorem}

Our main task in this section is to identify the limiting measure process $\mu_t$.
For simplicity, consider the case $\sigma = 1$. Recall that under the trivial initial condition, that is, $\lambda_i(0) = 0, i = 1, \dots, N$, the joint distribution of $\{\lambda_i(t)/t\}_{i=1}^N$ follows the (ordered) beta Laguerre ensemble. Thus, Theorem~\ref{thm:MKT-beta-ensembles} implies that $\cL(t \Gam(\alpha + c, 1)) = \cM_{c}(\mu_t)$, for $t > 0$. Here $\cL(t \Gam(\alpha + c, 1))$ means the distribution of a random variable $t X$, where $X$ has the gamma distribution $\Gam(\alpha+c,1)$.

A key observation is the following. Let $Y_t$ be a stochastic process satisfying 
\[
	dY_t = \sqrt{2Y_t} db_t + (\alpha + c) dt, \quad Y_0 = 0.
\]
It is worth noting that $Y_t$ is the 1d beta Laguerre process with parameter $(\alpha + c)$, which is known as the square of a Bessel process \cite[Chapter XI]{Revuz-Yor-book}. Thus, the distribution of $Y_t/t$ is $\Gam(\alpha + c, 1)$. We conclude that the distribution of $Y_t$ is the MKT of the limiting measure process under the trivial initial condition.

Here is our main result.
\begin{theorem}\label{thm:MKT-L}
Let $\xi$ be a random variable having distribution $\nu_0 = \cM_{c}(\mu_0)$. Let $Y_t$ be a stochastic process satisfying 
\begin{equation}\label{Yt}
	dY_t = \sqrt{2Y_t} db_t + (\alpha + c) dt, \quad Y_0 = \xi.
\end{equation}
Then
	$\cL(Y_t) = \cM_{c}(\mu_t)$. 
\end{theorem}
\begin{remark}
Let $Y^{(1)}_t$ be a stochastic process satisfying 
\[
	d Y^{(1)}_t = \sqrt{2Y^{(1)}_t} db^{(1)}_t + \left(\alpha + c -\frac12\right)dt, \quad Y^{(1)}_0 = 0,
\]
and 
\[
	Y_t = Y^{(1)}_t + \left(\sqrt \xi + \frac1{\sqrt2} b^{(2)}_t \right)^2.
\]
Here $b^{(1)}_t$ and $b^{(2)}_t$  are two independent standard Brownian motions which are also independent of $\xi$. Then $Y_t $ satisfies the SDE
\[
	dY_t = \sqrt{2 Y_t } db_t + (\alpha + c) dt
\]
with initial condition $Y_0 = \xi$ for some standard Brownian motion $b_t$. Note that the distribution of $Y^{(1)}_t$ is the MKT of the limiting process for beta Laguerre process with parameter $(\alpha - \frac12)$ under the trivial initial condition, provided that $\alpha > 1$.
This shows the similarity to the case where $\beta$ is fixed in which the limiting measure can be written as the free convolution involving the Marchenko--Pastur distribution, the initial measure and the semicircle distribution \cite{VW2022}. 
%\end{itemize}
\end{remark}

\subsubsection{An attempt on a direct method}
We use similar arguments as in the Gaussian case. Straightforward calculations will be omitted. First, with $f (x) = 1/(z-x)$, letting $N \to \infty$ in equation~\eqref{Ito-L}, we obtain a PDE for the Stieltjes transform 
\[
	\psi(t, z) = \bra{\mu_t, f(x)} = \int \frac{d\mu_t(x)}{z-x}, \quad  z \in \C_+, t \ge 0,
\]
as follows
\begin{equation}
	\psi_t = -c \psi^2 -2cz \psi \psi_z + (2-\alpha) \psi_z + z \psi_{zz}.
\end{equation}

Next, we derive the PDE for 
\[
	u(t, z) = \Ex[(z - Y_t)^{-c}], \quad z \in \C_+, t\ge 0.
\]
It is clear that
\[
	u_z = - c  \Ex[(z - Y_t)^{-c - 1}], \quad u_{zz} = c(c+1)  \Ex[(z - Y_t)^{-c - 2}].
\]
Now, using It\^o's formula, we get that
\begin{align*}
	d(z - Y_t)^{-c} &= \frac{c}{(z-Y_t)^{c+1}}dY_t + \frac{c(c+1)}{(z-Y_t)^{c+2}} Y_t dt\\
	&=\frac{c\sqrt{2Y_t}}{(z-Y_t)^{c+1}} db_t + \frac{c(\alpha + c)}{(z-Y_t)^{c+1}} dt+\frac{c(c+1)}{(z-Y_t)^{c+2}} [z - (z-Y_t)] dt\\
	&=\frac{c\sqrt{2Y_t}}{(z-Y_t)^{c+1}} db_t + \frac{c(\alpha - 1)}{(z-Y_t)^{c+1}} dt+\frac{z c(c+1)}{(z-Y_t)^{c+2}}  dt.
\end{align*}
This implies the following PDE for $u$,
\[
	u_t = (1-\alpha) u_z + z u_{zz}.
\]

We claim that the function $M$ defined by
\[
	M(t, z) = -\frac1c (\log(u))_z,
\]
satisfies the same PDE as $\psi$. Indeed, we calculate steps by steps as
\[
	M_t = -\frac1c \left(\frac{u_t}{u}\right)_z = -\frac1c \left(\frac{(1-\alpha) u_z + z u_{zz}}{u}\right)_z = (1-\alpha)M_z - \frac1c \left(\frac{zu_{zz}}{u}\right)_z,
\]
\[
	M_z = -\frac1c \left(\frac{u_z}{u}\right)_z =  -\frac1c \frac{u u_{zz} - u_z^2}{u^2} =  -\frac1c \frac{u_{zz}}{u} + cM^2,
\]
and
\[
	(zM_z)_z = - \frac1c \left(\frac{zu_{zz}}{u}\right)_z + (cz M^2)_z = M_t +(\alpha - 1)M_z +2cz M M_z + cM^2.
\]
Then we get
\[
	M_t = - cM^2 - 2cz M M_z +(2 - \alpha)M_z + z M_{zz},
\]
which is exactly the same as the PDE for $\psi$. The proof would be similar to the Gaussian case if we could show that the above PDE admits a unique solution, which we do not know how to deal with yet.

\subsubsection{A formal series approach}
We prove Theorem~\ref{thm:MKT-L} by using a formal series approach. We first derive the ODEs for moments of $Y_t$.
Define 
\[
	h_n(t) = \Ex[Y_t^n]. 
\]
By It\^o's formula, we get that
\begin{align*}
	d Y_t^n &= n Y_t^{n-1} dY_t + n(n-1)Y_t^{n-1}dt\\
	&=n Y_t^{n-1} [\sqrt{2Y_t} db_t + (\alpha + c) dt] + n(n-1)Y_t^{n-1}dt\\
	&=n Y_t^{n-1} \sqrt{2Y_t} db_t  + n(\alpha + c + n - 1) Y_t^{n-1}dt.
\end{align*}
It follows that 
\[
	h_n'(t) = n(\alpha + c + n - 1) h_{n-1}(t), \quad h_n(0) = \Ex[\xi^n].
\]
Then $h_n(t)$ is a polynomial in $t$.

Next, we define the formal power series
\[
	\phi(t, z) = \sum_{n=0}^\infty \frac{(c)_n}{n!} h_n(t) \frac{1}{z^{c+n}},
\]
which has a similar role as the generalized Stieltjes transform $u(t, z)$. We calculate formal partial derivatives of $\phi$ as
\[
	\phi_t = \sum_{n=0}^\infty \frac{(c)_n}{n!} h_n'(t) \frac{1}{z^{c+n}} =  \sum_{n=1}^\infty \frac{(c)_n}{(n-1)!} (\alpha + c + n - 1) h_{n-1}(t) \frac{1}{z^{c+n}},
\]
and
\[
	\phi_z = -\sum_{n=0}^\infty \frac{(c)_{n+1}}{n!} h_n(t) \frac{1}{z^{c+n+1}}, \quad \phi_{zz} = \sum_{n=0}^\infty \frac{(c)_{n+2}}{n!} h_n(t) \frac{1}{z^{c+n+2}}.
\]
Then ODEs for moments $\{h_n(t)\}$ are equivalent to the formal PDE
\[
	\phi_t = (1-\alpha)\phi_z + z \phi_{zz}.
\]

Recall that the limiting moment processes are recursively defined by $m_0(t) = 1,$
\begin{align}
	\frac{dm_n(t)}{dt} =  n (n+\alpha - 1) m_{n-1}(t) +  {cn} \sum_{j=0}^{n-1} m_j(t) m_{n-1-j}(t), \quad m_n(0) = a_n.
\end{align}
Consider the formal power series
\[
	\psi(t, z) = \sum_{n=0}^{\infty} {m_n(t)} \frac1{ z^{n+1}}.
\]
We claim that the recurrence relations for $m_n(t)$ are equivalent to the following formal PDE
\[
	\psi_t = - c\psi^2 - 2cz \psi \psi_z +(2 - \alpha)\psi_z + z \psi_{zz}.
\]
We omit the proof of this claim because it is just a direct calculation. 

Now we use the same idea as in the Gaussian case. 
Define $\tilde m_n(t)$ by the moment relation~\eqref{moment-MKR} with $h_n(t)$, then define the formal series
\[
	\tilde\psi(t, z) = \sum_{n=0}^{\infty} {\tilde m_n(t)} \frac1{ z^{n+1}}.
\]
Then $\tilde\psi$ and $\phi$ are related by
\[
	\tilde\psi = -\frac1c \frac{\phi_z}{\phi}.
\]
From which, we deduce that $\tilde\psi$ satisfies 
\[
	\tilde\psi_t = - c\tilde\psi^2 - 2cz \tilde\psi \tilde\psi_z +(2 - \alpha)\tilde\psi_z + z \tilde\psi_{zz}.
\]
This PDE implies ODEs for $\tilde m_n(t)$ which are exactly the same as those for $m_n(t)$. By the construction, $m_n(t)$ and $\tilde m_n(t)$ have the same initial condition. Consequently, we conclude that $m_n(t) = \tilde m_n(t)$. Therefore, the law of $Y_t$ is the MKT of $\mu_t$ by taking into account Theorem~\ref{thm:MKT-moment}. The proof is complete.\qed

\subsection{Local scaling regime}
Let $0 \le \lambda_1(t) \le \lambda_2(t) \le \cdots \le \lambda_N(t)$ be the beta Laguerre process satisfying SDEs~\eqref{bLP} in a high temperature regime $\beta = 2c/N$.
We consider a local scaling by changing of variables: $x_i(t) = \gamma (\lambda_i(\frac t \tau) - E)$, where $\gamma = \gamma_N, \tau = \tau_N, E = E_N > 0$ are non-random. We aim to show that under suitable scaling, the empirical measure process of $\{x_i(t)\}$ converges to the limit of the Gaussian case.

We begin with writing down the SDEs for $x_i(t)$
\begin{align*}
	dx_i(t) &= \gamma d\lambda_i(\tfrac t  \tau)\\
	&=\gamma \left \{ \sqrt{2 \lambda_i(\tfrac t \tau)} db_i(\tfrac t\tau) + \frac{\alpha}{\tau} dt + \frac{c}{\tau N} \sum\limits_{j : j \neq i} \dfrac{2(x_i(t) + \gamma E)}{x_i(t) - x_j(t)} dt \right\} \\
	&= \sqrt{2\frac{\gamma}{\tau}x_i(t) + \frac{2\gamma^2 E}{\tau}} dw_i(t) + \frac{\gamma}{\tau} \alpha dt + \frac{\gamma}{\tau} \frac{c}{N} \sum\limits_{j : j \neq i} \dfrac{2x_i(t)}{x_i(t) - x_j(t)} dt \\
	&\quad + \frac{2\gamma^2 E} {\tau} \frac{c}{N} \sum\limits_{j : j \neq i} \dfrac{1}{x_i(t) - x_j(t)} dt.
\end{align*}
Here $\{w_i(t) = \sqrt \tau b_i(\frac t \tau)\}_{i=1}^N$ are also independent standard Brownian motions. In the regime where $\gamma \to \infty, \tau \to \infty$ with 
\[
	\frac{\gamma}{\tau} \to 0, \quad \frac{2\gamma^2E}{\tau} \to \sigma^2,
\]
heuristically, the SDEs for $x_i(t)$ are approximated by 
\[
	dx_i(t) = \sigma dw_i(t) +\sigma^2 \frac{c}{N} \sum\limits_{j : j \neq i} \dfrac{1}{x_i(t) - x_j(t)} dt.
\]
It suggests the same limiting behavior as in the Gaussian case. We are going to use a moment method to prove it rigorously.

For simplicity, we denote 
\[
	\epsilon_N = \frac{\gamma}{\tau}, \quad \sigma_N^2 = \frac{2\gamma^2E}{\tau}.
\]
In the considering regime, 
\[
	\epsilon_N \to 0, \quad \sigma_N^2 \to \sigma^2.
\]
We rewrite the SDEs for $x_i$ as
\begin{align*}
	dx_i 
	&= \sqrt{2\epsilon_N x_i + \sigma_N^2} dw_i + \epsilon_N \left\{ \alpha dt +  \frac{c}{N} \sum\limits_{j : j \neq i} \dfrac{2x_i}{x_i - x_j} dt \right\} \\
	&\quad + \sigma_N^2 \frac{c}{N} \sum\limits_{j : j \neq i} \dfrac{1}{x_i - x_j} dt.
\end{align*}
We use the same notations for the empirical measure process and the moment processes. By a straightforward calculation, we obtain that 
\begin{align}
	&d S_n^{(N)}(t) = dM_n^{(N)}(t)  \notag\\
	&\quad + \epsilon_N \left\{ n(n + \alpha - 1) S_{n-1}^{(N)}(t) dt + c n \sum_{i = 0}^{n-1} S_i^{(N)}(t) S_{n-i-1}^{(N)}(t) dt - \frac{c n^2}N S_{n-1}^{(N)}(t) dt\right\} \notag \\
	&\quad + \sigma_N^2 \bigg\{\frac{c n}2 \sum_{j = 0}^{n-2} S_{j}^{(N)}(t) S_{n-2-j}^{(N)}(t) dt + \frac12 n(n-1) S_{n-2}^{(N)}(t) dt \notag\\
	&\qquad\qquad - \frac c{2N} n(n-1) S_{n-2}^{(N)}(t) dt\bigg\}.\label{L-local-moments}
\end{align}
Here $M_n^{(N)}(t)$ is the martingale part with the quadratic variation
\[
	[M_n^{(N)}](t) = \epsilon_N \left(  \frac{2 n^2}{N}   \int_0^t \frac{\sum_{i = 1}^N x_i^{2n-1}}{N}ds\right) + \sigma_N^2 \left( \frac{ n^2}{N} \int_0^t \frac{\sum_{i=1}^N x_i^{2n-2}}{N} ds\right).
\]
Intuitively, those equations are the sum of two parts: the Gaussian part with parameter $\sigma = \sigma_N$ and the Laguerre part with parameter $\sigma = \epsilon_N$. In the regime where $\epsilon_N \to 0$ and $\sigma_N^2 \to \sigma^2$, the Laguerre part vanishes and we get the following limiting behavior.

\begin{theorem}\label{thm:local-LG}
	Consider (standard) beta Laguerre processes $\{\lambda_i(t)\}_{i=1}^N$ in a high temperature regime where $\beta = {2c}/N$ with a positive constant $c \in (0, \infty)$. Let 
\[
	x_i(t) = \gamma_N \left(\lambda_i \left(\frac t {\tau_N}\right) - E_N\right),
\]
where $\gamma_N, \tau_N, E_N > 0$ are non-random. Let $\mu_t^{(N)}$ be the empirical measure process of $\{x_i(t)\}_{i=1}^N$ and $S_n^{(N)}(t)$ be its $n$th moment process. As $N \to \infty$ with
\[
	\epsilon_N = \frac{\gamma_N}{\tau_N} \to 0, \quad \text{and}\quad \sigma_N^2 = \frac{2\gamma_N^2E_N}{\tau_N} \to \sigma^2,
\]
for given $\sigma^2 > 0$, we get exactly the same statements as in Theorem~{\rm\ref{thm:G-moments}} and Theorem~\rm{\ref{thm:G-measures}}.
\end{theorem}

\begin{proof}
Equation~\eqref{L-local-moments} enables us to prove the above theorem by induction, provided that the martingale part converges to zero, which is shown in the next lemma. 
\end{proof}
\begin{lemma}
Assume the moment condition~\eqref{moment-assumption}. Then in the regime where $\epsilon_N \to 0$ and $\sigma_N^2 \to \sigma^2$, there is a constant $C_n > 0$ such that
		\begin{equation}\label{bounded-absolute-moment}
			\Ex \left [\frac1N \sum_{i=1}^N |x_i(t)|^n \right] \le C_n,
		\end{equation}
for all $N$ and all $t \in [0,T]$. Consequently, in that regime, the martingale part $M_n^{(N)}(t)$ converges to zero in probability (as random elements in $C([0, T])$).
\end{lemma}
\begin{proof}
We prove the lemma by induction. The case $n=0$ is trivial. Let $n \ge 2$ be even. Assume that 
\[
	\Ex \left [\frac1N \sum_{i=1}^N |x_i|^{n-2} \right] = \Ex[S_{n-2}^{(N)}(t)] \le C_{n-2},
\]
for all $N$ and all $t \in [0,T]$. Our aim is to estimate $\Ex[S_n^{(N)}(t)]$.

Write the equation~\eqref{L-local-moments} in the integral form, then take the expectation, we obtain that
\begin{align*}
	&\Ex[S_n^{(N)}(t)] = a_n  \\
	&\quad + \epsilon_N \Ex \int_0^t\bigg \{ n(n + \alpha - 1) S_{n-1}^{(N)}(u) + c n \sum_{i = 0}^{n-1} S_i^{(N)}(u) S_{n-i-1}^{(N)}(u) - \frac{c n^2}N S_{n-1}^{(N)}(u)\bigg\} du  \notag \\
	&\quad + \sigma_N^2 \Ex \int_0^t\bigg \{\frac{c n}2 \sum_{j = 0}^{n-2} S_{j}^{(N)}(u) S_{n-2-j}^{(N)}(u) + \frac12 n(n-1) S_{n-2}^{(N)}(u) \\
	&\qquad \qquad\qquad - \frac c{2N} n(n-1) S_{n-2}^{(N)}(u) \bigg\} du.
\end{align*}
We use the following estimates
\begin{align*}
	S_{j}^{(N)}(u) S_{n-2-j}^{(N)}(u) \le S_{n-2}^{(N)}(u),\quad  (j = 0,1,\dots, n-2),\\
	S_j^{(N)}(u) S_{n-1-j}^{(N)}(u) \le\frac1N \sum_{i=1}^N |x_i|^{n-1} =: |S|_{n-1}^{(N)}(u), \quad (0 \le j \le n-1),
\end{align*}
to deduce that 
\begin{align}
0 \le \Ex[S_n^{(N)}(t)] &\le S_n^{(N)}(0)  +  \epsilon_N \left(n(n+\alpha-1) + cn^2 + \frac{cn^2}N  \right) \int_0^t \Ex[|S|_{n-1}^{(N)}(u)] du \notag\\
&\quad + \sigma_N^2 \frac12 (c+1)n(n-1) \int_0^t \Ex[S_{n-2}^{(N)}(u)] du.\label{moment-mean-bounded}
\end{align}

Let $\tau \in [0,T]$ be a point such that $\Ex[|S|_{n-1}^{(N)}(u)]$ attains the maximum value in $[0,T]$. Then for such $\tau$, 
\[
	\int_0^\tau \Ex[|S|_{n-1}^{(N)}(u)] du \le \tau \Ex[|S|_{n-1}^{(N)}(\tau)] \le T\Ex[|S|_{n-1}^{(N)}(\tau)].
\]
Using the inequality $|x|^{n-1} \le x^n + 1$, for $x \in \R$, we obtain that
\[
	|S|_{n-1}^{(N)}(\tau) \le  S_n^{(N)}(\tau) + 1,
\]
which implies
\begin{equation}\label{n-1-n}
	\Ex[|S|_{n-1}^{(N)}(\tau)] \le  \Ex[S_n^{(N)}(\tau)] + 1.	
\end{equation}
The inequality~\eqref{moment-mean-bounded} at $t = \tau$ gives us 
\begin{align*}
0 \le \Ex[S_n^{(N)}(\tau)] &\le S_n^{(N)}(0) + \epsilon_N \left(n(n+\alpha-1) + cn^2 + \frac{cn^2}N \right) T (\Ex[S_n^{(N)}(\tau)] + 1) \notag\\
&\quad + \sigma_N^2 \frac12 (c+1)n(n-1) \int_0^\tau \Ex[S_{n-2}^{(N)}(u)] du.
\end{align*}
Therefore, when $\epsilon_N$ is small, say
\[
	\epsilon_N \left(n(n+\alpha-1) + cn^2 + \frac{cn^2}N \right) T < \frac12,
\]
we obtain that 
\begin{align*}
0 \le\frac12\Ex[S_n^{(N)}(\tau)] &\le S_n^{(N)}(0) + \frac12  + \sigma_N^2 \frac12 (c+1)n(n-1) \int_0^\tau \Ex[S_{n-2}^{(N)}(u)] du\\
&\le S_n^{(N)}(0) + \frac12 +\frac12 \sigma_N^2  (c+1)n(n-1) T C_{n-2},
\end{align*}
implying that $\Ex[S_n^{(N)}(\tau)] \le D_n$, for some constant $D_n$ depending only on $n$. Now it follows from equation~\eqref{n-1-n} that for any $t \in [0, T]$,
\[
	\Ex[|S|_{n-1}^{(N)}(t) ] \le \Ex[|S|_{n-1}^{(N)}(\tau) ] \le	\Ex[S_n^{(N)}(\tau) ] + 1 \le D_n + 1,
\]
meaning that the estimate~\eqref{bounded-absolute-moment} holds for $n-1$ with $C_{n-1} = D_n + 1$. That estimate for $\Ex[|S|_{n-1}^{(N)}(t) ]$, together with equation~\eqref{moment-mean-bounded}, implies the desired estimate for $\Ex[S_n^{(N)}(t)]$. The martingale part converges to zero as a consequence of those estimates. The proof is complete.
\end{proof}

\section{Beta Jacobi processes}\label{sect:J}
In the Jacobi case, eigenvalue processes are beta Jacobi processes $0 \le \lambda_1(t) \le \lambda_2(t) \le \cdots \le \lambda_N(t) \le 1$ which satisfy the following SDEs \cite{Demni-2010}
\begin{equation}\label{SDEs-J}
	d \lambda_i =  \sqrt{2\lambda_i (1 - \lambda_i)} db_i + \left(a + 1 - (a+b+2) \lambda_i + \frac\beta2\sum_{j: j \neq i} \frac{2 \lambda_i(1 - \lambda_i)} {\lambda_i - \lambda_j} \right) dt, 
\end{equation}
$(i=1,\dots, N)$. Here $a, b > -1/2$, $\beta > 0$ and $\{b_i(t)\}_{i=1}^N$ are independent standard Brownian motions. Different from the Gaussian and Laguerre cases, the (ordered) beta Jacobi ensemble is the stationary distribution of the above eigenvalue processes. For $\beta \ge 1$, it is known that $\{\lambda_i(t)\}_{i=1}^N$ never collide for $t > 0$ \cite{Graczyk-Malecki-2014}.

\subsection{Mean-field limit for beta Jacobi ensembles}
In a high temperature regime $\beta = 2c/N$, the convergence to a limit of the empirical measure processes has been established by a moment method \cite{Trinh-Trinh-Jacobi}. We recall the result here. Let $\mu_t^{(N)}$ be the empirical measure process of $\{\lambda_i(t)\}_{i=1}^N$. In this case, all probability measures $\mu_t^{(N)}$ are supported on $[0,1]$. Thus, the weak convergence is equivalent to the convergence of all moments. It was shown in \cite{Trinh-Trinh-Jacobi} that 
\begin{proposition}[\cite{Trinh-Trinh-Jacobi}]
	Assume that the initial measure $\mu_0^{(N)}$ converges weakly to a probability measure $\mu_0$. Then the following hold.
	\begin{itemize}
		\item[\rm (i)] For each $n = 1, 2,\dots,$ the sequence of moment processes $S_n^{(N)}(t) = \bra{\mu_t^{(N)}, x^n}$ converges in probability as random elements in $C([0, T])$ to a deterministic limit $m_n(t)$, which is defined inductively as the solution to the following initial value ODE
		\begin{align*}
			&m_n'(t) = -n(2c+a+b+n+1) m_n(t) + n(a+n)m_{n-1}(t)\\
			&\qquad\qquad + cn \sum_{i=0}^{n-1}m_i(t) m_{n-1-i}(t) - cn \sum_{j=1}^{n-1} m_j(t) m_{n-j}(t),\\
			&m_n(0) = \lim_{N\to \infty} \bra{\mu_0^{(N)}, x^n}, \quad (n \ge 1).
		\end{align*}
Here $m_0(t) \equiv 1$.

		\item[\rm(ii)] For any $t \ge 0$, let $\mu_t$ be the unique probability measure with moments $\{m_n(t)\}$. Then the sequence $\mu_t^{(N)}$ converges in probability to $(\mu_t)_{0 \le t \le T}$ as random elements in the space of continuous maps from $[0,T]$ to the space of probability measures on $[0,1]$ endowed with the uniform topology.
	\end{itemize}
\end{proposition}

Motivated by the result relating the beta distribution and the limiting distribution of beta Jacobi ensembles in  Theorem~\ref{thm:MKT-beta-ensembles}(iii), we identify the limiting measure process $\mu_t$ as follows.
\begin{theorem}
Let $\xi$ be a random variable with distribution $\nu_0 = \cM_{c}(\mu_0)$. Let $Y_t$ be a process satisfying the SDE
\begin{equation}
	dY_t = \sqrt{2 Y_t(1-Y_t)} db_t + (a+c+1)dt - (a+b+2c+2) Y_t dt,\quad Y_0 = \xi.
\end{equation}
Then the distribution of $Y_t$ is the MKT of $\mu_t$ with parameter $c$.
\end{theorem}

We omit the proof because it is similar to the previous two cases. We note here two PDEs for 
\(
	\psi(t, z) = \bra{\mu_t, (z-x)^{-1}},\) and  \(\phi(t, z) = \Ex[(z - Y_t)^{-c}],
\)
\begin{align}
	\psi_t &= (a+b)\psi + (1-a)\psi_z + (a+b-2)z\psi_z + z(1-z) \psi_{zz} \nonumber\\
	&\quad + c((2z - 1)\psi^2 -2z(1-z)\psi \psi_z),
\end{align}
and
\begin{equation}
	\phi_t = c(a+b+c+1)\phi -a \phi_z + (a+b)z\phi_z + z(1-z)\phi_{zz}.
\end{equation}
The former can be obtained from the latter by taking the change of variables $\psi = -{\phi_z}/(c{\phi})$.
\subsection{Gaussian regime}
We turn to study local scaling of beta Jacobi processes. When $\beta$ is fixed, such problem has been studied in \cite{AVW2024}. We begin with a local scaling $x_i(t) = \gamma (\lambda_i(\frac t \tau) - E)$, where $\gamma, \tau, E > 0$ are non-random. The SDEs for $x_i(t)$ become
\begin{equation}\label{J-local}
	dx_i =\sqrt{F(x_i)} dw_i + (U + Vx_i) dt 
	+ \frac cN\sum_{j: j \neq i} \frac{F(x_i)} {x_i - x_j}  dt.
\end{equation}
Here $w_i(t) = \sqrt{\tau} b_i(\frac t\tau)$ are standard Brownian motions, and 
\begin{align*}
	U &=  \frac{\gamma}{\tau}\{ (a+1) - (a+b+2)E  \},\\
	V &= -  \frac{1}{\tau}(a+b+2),\\
	F(x_i)&=A x_i^2 + B x_i + C, \text{with} \begin{cases}
	A = -\frac2{\tau},\\
	B =  \frac{2\gamma}{\tau} (1 - 2E),\\
	C = \frac{2\gamma^2}{\tau} E(1-E).
	\end{cases}
\end{align*}

Now, let us consider the following conditions
\[
	\tau \to \infty, \quad \frac{\gamma}{\tau} \to 0, \quad \frac{2\gamma^2}{\tau} E(1-E) \to \sigma^2,
\]
with $\sigma^2 > 0$. Under those conditions, it is clear that 
\[
	U \to 0, \quad V\to 0, \quad A \to 0, \quad B \to 0, \quad C \to \sigma^2.
\]
Intuitively, the SDEs for $\{x_i\}$ can be approximated by 
\[
	dx_i(t) = \sigma dw_i(t) +\sigma^2 \frac{c}{N} \sum\limits_{j : j \neq i} \dfrac{1}{x_i(t) - x_j(t)} dt.
\]
Then we can use a moment method to show exactly the same statements as in Theorem~\ref{thm:G-moments} and Theorem~\ref{thm:G-measures}. 

\begin{theorem}\label{thm:local-JG}
Let $\{\lambda_i(t)\}_{i=1}^N$ be the beta Jacobi processes~\eqref{SDEs-J}. Consider a local scaling $x_i(t) = \gamma_N ( \lambda_i(\frac t {\tau_N}) - E_N)$. Denote by $\mu_t^{(N)}$ the empirical measure process of $\{x_i(t)\}$ and $S_n^{(N)}(t)$ its $n$th moment process. Assume that 
\[
	\tau_N \to \infty, \quad \frac{\gamma_N}{\tau_N} \to 0, \quad \frac{2\gamma_N^2}{\tau_N} E_N(1-E_N) \to \sigma^2,
\]
where $\sigma^2 > 0$ is given. Then all statements as in Theorem~{\rm\ref{thm:G-moments}} and Theorem~{\rm\ref{thm:G-measures}} hold.
\end{theorem}

\subsection{Laguerre regime}
Next, we consider an edge scaling case where
\[
	E = 0,\quad \tau \to \infty,\quad \frac{\gamma}{\tau} \to \sigma.
\] 
Here $\sigma > 0$ is given.
Under these conditions, 
\[
	U \to \sigma (a + 1), \quad V\to 0, \quad A \to 0, \quad B \to 2\sigma, \quad C = 0.
\]
Thus
\[
	F(x_i) \approx 2\sigma x_i,
\]
and the SDEs for $\{x_i\}_{i=1}^N$ can be approximated by 
\[
	dx_i = \sqrt \sigma \sqrt{2x_i}  d  w_i   + \sigma \alpha  d  t + \sigma\frac c{N} \sum\limits_{j : j \neq i} \dfrac{2x_i}{x_i - x_j} dt,
\]
with $\alpha = a + 1$. We arrive at the following result 
\begin{theorem}\label{thm:local-JL}
Let $\{\lambda_i(t)\}_{i=1}^N$ be the beta Jacobi processes~\eqref{SDEs-J}. Consider a local scaling $x_i(t) = \gamma_N \lambda_i(\frac t {\tau_N})$. Denote by $\mu_t^{(N)}$ the empirical measure process of $\{x_i(t)\}$ and $S_n^{(N)}(t)$ its $n$th moment process. Assume that 
\[
	\tau_N \to \infty,\quad \frac{\gamma_N}{\tau_N} \to \sigma,
\]
where $\sigma > 0$ is given. Then all statements as in Theorem~{\rm\ref{thm:L}} hold.
\end{theorem}
We skip the proofs of Theorem~\ref{thm:local-JG} and Theorem~\ref{thm:local-JL} because they are similar to the proof of Theorem~\ref{thm:local-LG}.

\begin{acknowledgments}
The authors would like to thank referees for valuable comments and suggestions.
This research is funded by Vietnam National Foundation for Science and Technology Development (NAFOSTED) under grant number 101.03-2021.29 and is supported by JSPS KAKENHI Grant numbers 20K03659 (F.N.) and JP24K06766 (K.D.T.). 
\end{acknowledgments}

%

%\bibliography{bib}% Produces the bibliography via BibTeX.

\end{document}